\documentclass[twoside, 10pt]{article}
\usepackage{amsmath,amsthm,amssymb,amscd,amstext,amsfonts}
\usepackage[utf8]{inputenc}
\usepackage{mathtools}
\usepackage{mathrsfs}
\usepackage{thmtools}
\usepackage{latexsym}
\usepackage{verbatim}
\usepackage{framed}
\usepackage{graphicx}
\usepackage{stmaryrd}
\usepackage{enumitem}
\usepackage{fullpage}
\usepackage{bm}
\usepackage{url}
\usepackage{physics}
\usepackage{dsfont}
\usepackage{todonotes}
\usepackage{tikz}
\usepackage{graphicx}
\usepackage{titling}
\usepackage{color}
\usepackage{xpatch}
\usepackage{algorithm}
\usepackage{algorithmic}
\usepackage{epsf}

\usepackage[small]{titlesec}
\titlelabel{\thetitle.\enspace}

\usepackage[small, labelfont=bf, labelsep=period]{caption}

\usepackage[breaklinks=false, linktocpage=true,]{hyperref}

\hypersetup{
	colorlinks = true,
    linkcolor={purple},
    urlcolor={purple},
    citecolor={blue!80!black}
}

\makeatletter
\def\tagform@#1{\maketag@@@{(\ignorespaces{\oldstylenums{#1}}\unskip\@@italiccorr)}}
\renewcommand{\eqref}[1]{\textup{{\normalfont(\oldstylenums{\ref{#1}}}\normalfont)}}
\newcommand{\ps@bookheader}{%
\renewcommand\@oddfoot{\hfil}%
\renewcommand\@evenfoot{\hfil}%
\renewcommand\@oddhead{\ifnum\value{page}>1
{\small\hfil GOH AND HATAMI}\hfil\thepage\else\hfil\fi}}
\renewcommand\@evenhead{\thepage\hfil\small BLOCK COMPLEXITY AND IDEMPOTENT SCHUR MULTIPLIERS\hfil}%
\makeatother
\usepackage[letterpaper, total={5.5in, 8in}, vmarginratio=1:1, hmarginratio=1:1, headsep=21pt]{geometry}

\usepackage[nameinlink, noabbrev, capitalize]{cleveref}

\declaretheoremstyle[bodyfont=\normalfont\slshape, notefont=\normalfont\itshape,notebraces={{\rm(}}{{\rm)}}, postheadspace=0.5em,headpunct={\rm.}, spaceabove=8pt, spacebelow=8pt]{slbody}

\declaretheorem[name=Theorem, numberwithin=section, style=slbody]{theorem}
\declaretheorem[name=Lemma, numberwithin=section, sibling=theorem, style=slbody]{lemma}

\declaretheorem[name=Proposition, numberwithin=section, sibling=theorem, style=slbody]{proposition}
\declaretheorem[name=Conjecture, numberwithin=section, sibling=theorem, style=slbody]{conjecture}

\renewcommand\norm[1]{\left|\!\left|#1\right|\!\right|}
\newcommand\normm[1]{\biggl|\!\biggl|#1\biggr|\!\biggr|}

\newcommand\bignorm[1]{\bigl|\!\bigl|#1\bigr|\!\bigr|}
\newcommand\normmax[1]{|\!|#1|\!|_{\rm max}}

\newcommand\normrow[1]{\left|\!\left|#1\right|\!\right|_{\rm row}}
\newcommand\normcol[1]{\left|\!\left|#1\right|\!\right|_{\rm col}}
\newcommand\normgamma[1]{|\!|#1|\!|_{\gamma_2}}

\newcommand\opnorm[1]{\left|\!\left|#1\right|\!\right|_{\rm op}}
\newcommand\normalg[1]{\left|\!\left|#1\right|\!\right|_{\rm A}}
\newcommand\normmultiplier[1]{\left|\!\left|#1\right|\!\right|_{\rm m}}

\newcommand\eps{\epsilon}
\renewcommand\tilde{\widetilde}

\renewcommand\hat{\widehat}

\DeclareMathOperator{\block}{block}
\DeclareMathOperator{\Ldim}{Ldim}
\DeclareMathOperator{\sgn}{sgn}

\newcommand{\RR}{\mathbf{R}}   
\newcommand{\CC}{\mathbf{C}}   
\newcommand{\ZZ}{\mathbf{Z}}   
\newcommand{\NN}{\mathbf{N}}   

\newcommand{\one}{\mathop{\mathbf{1}}\nolimits}
\newcommand\matr[4]{\big(\genfrac{}{}{0pt}{1}{#1}{#3}\genfrac{}{}{0pt}{1}{#2}{#4}\big)}
\newcommand\ex{\mathop{\mathbf{E}}\nolimits}

\newcommand\alg[2]{\medbreak\noindent{\bf#1} ({\it#2\/}).\enspace\ignorespaces}

\renewcommand{\maketitle}{%
  \begin{center}
    {\large\bf Block complexity and idempotent Schur multipliers}\\
    \vskip 36pt
    {\sc Marcel K.\ Goh\enspace{\rm and}\enspace Hamed Hatami}
    \medskip
    \vskip 36pt
  \end{center}
}
\title{}\author{}\date{}

\begin{document}

\maketitle
\renewenvironment{abstract}{\quotation\noindent\small{\bfseries\abstractname.\enspace}}{\endquotation}

\begin{abstract}
We call a matrix blocky if, up to row and column permutations, it can be obtained from an identity matrix by repeatedly applying one of the following operations: duplicating a row, duplicating a column, or adding a zero row or column. Blocky matrices are precisely the boolean matrices that are contractive when considered as Schur multipliers. It is conjectured that any boolean matrix with Schur multiplier norm at most $\gamma$ is expressible as a signed sum 
\begin{equation*}A = \sum_{i=1}^L \pm B_i\end{equation*}
for some blocky matrices $B_i$, where $L$ depends only on $\gamma$. This conjecture is an analogue of Green and Sanders's quantitative version of Cohen's idempotent theorem. In this paper, we prove bounds on $L$ that are polylogarithmic in the dimension of $A$. Concretely, if $A$ is an $n\times n$ matrix, we show that one may take $L = 2^{O(\gamma^7)} \log(n)^2$. 
\end{abstract}

\vskip50pt
\baselineskip=13.5pt

\section{Introduction} 

A celebrated classical result in harmonic analysis is the 1960 theorem of P.~J.~Cohen~\cite{cohen1960}, which states that a measure $\mu$ on a locally compact abelian group $G$ is idempotent if and only if $\hat\mu$ is a member of the coset ring of $\hat G$. This was conjectured by W.\ Rudin \cite{rudin1959}, who settled the case of discrete groups and finite-dimensional torus groups. An explicit description of idempotent measures on the circle group was given earlier by H.\ Helson \cite{helson1954}. Cohen's idempotent theorem is sometimes referred to as the Cohen--Host idempotent theorem, as in 1986, B.~Host extended Cohen's result to arbitrary locally compact groups~\cite{host1986}.

For finite groups, Cohen's theorem holds without offering any content, since in this case, the coset ring consists of all functions on $\hat G$. However, a 2008 theorem of B.~Green and T.~Sanders~\cite{greensandersannals} gives a quantitative version of the idempotent theorem that is useful even for finite groups.
Subsequent refinements by Sanders~\cite{sanders2011, sanders2020} give better bounds and show that the group need not even be abelian for this result to hold. For the sake of simplicity, in the brief exposition below, we shall take $G$ to be abelian and finite.

Let $\normalg{f}$ denote the Fourier algebra norm $\sum_{a\in \hat G} \bigl|\hat f(a)\bigr|$. A classical result of Kawada and It\^o~\cite{kawadaito} states that a nonzero boolean function $f:G \to \{0,1\}$ satisfies $\normalg{f}  \le 1$ if and only if $f=\one_{s+H}$ for a coset $s+H$ of $G$. The quantitative idempotent theorem states that every  $f:G \to \{0,1\}$ with $\normalg{f} \le \gamma$ can be expressed as a signed sum 
\begin{equation}
\label{eq:green_sanders_sum}
f = \sum_{i=1}^L \pm \one_{s_i + H_i}
\end{equation}
for some $L \le \exp(\gamma^{3+o(1)})$ and some cosets $s_1 + H_1,\ldots, s_L + H_L$ of $G$. The term ``idempotent'' refers to the fact that boolean functions are precisely those that satisfy $f^2=f$, and therefore, Green and Sanders' theorem characterizes the idempotents that have small Fourier algebra norm.

\paragraph{Blocky matrices and Schur multipliers.}\hskip-0.3em%
The aim of this paper is to investigate an analogue of the program surrounding Cohen’s idempotent theorem in the setting of Schur multipliers, which we now introduce.

Let $\ell^2$ denote the Hilbert space of all square summable complex sequences, and let $B(\ell^2)$ be the collection of all bounded linear operators $A:\ell^2 \to \ell^2$ equipped with its operator norm 
\[\opnorm{A} = \sup_{\substack{y\in \ell^2 \\ y \neq 0}} \frac{\norm{Ay}_2}{\norm{y}_2}.\]  
Every operator $A \in B(\ell^2)$ is uniquely identified by its associated matrix $(a_{i,j})_{i,j \in \NN}$ defined by 
$a_{i,j} = \bigl\langle A e_i, e_j\bigr\rangle$, where $\{e_i\}_{i \in \NN}$ is the standard orthonormal basis of $\ell^2$.


Every matrix $M : \NN \times \NN \to \CC$  gives rise to a linear transformation on the space of all matrices $A:\NN \times \NN \to \CC$ via the map $S_M:A \mapsto M \circ A$, where $\circ$ denotes the entrywise (Schur) product. The matrix $M$ is called a \emph{Schur multiplier} if $S_M$ maps $B(\ell^2)$ into itself, that is $M\circ A \in B(\ell^2)$ for every  $A \in B(\ell^2)$. Equivalently, $M$ is a Schur multiplier if its \emph{Schur multiplier norm}, defined by
\begin{equation} \normmultiplier{M}= \sup_{\substack{A\in B(\ell^2) \\ A \neq 0}} \frac{\opnorm{M\circ A}}{\opnorm A},
\end{equation}
is finite. Since the Schur multiplier norm satisfies  
\begin{equation}\normmultiplier{A\circ B} \le \normmultiplier{A} \cdot \normmultiplier{B},\label{eqgammaalgebra}
\end{equation}
the set of Schur multipliers is closed under addition and Schur product, and thus forms a Banach algebra. 

An element $a$ of an algebra is said to be \emph{idempotent} if $a^2 = a$. Any matrix satisfying $M\circ M = M$ must be boolean, but not all infinite boolean matrices are Schur multipliers. A boolean Schur multiplier $M: \NN \times \NN \to \{0,1\}$ corresponds to the natural operation that maps any matrix $A$ to a new matrix that agrees with $A$ on the support of $M$ and is zero everywhere else. The following is the central question of interest in this article: 
\begin{quote}
\textsl{What are the idempotent elements of the Schur multiplier algebra?}
\end{quote}
As a starting point, we discuss a simple characterization of the \emph{contractive} idempotent elements, i.e., those with Schur multiplier norm at most 1.  

We call a boolean matrix $B:\NN \times \NN \to \{0,1\}$ \emph{blocky} if there exist families $\{S_i\}_{i\in \NN}$ and $\{T_i\}_{i\in \NN}$ of pairwise disjoint subsets of $\NN$ such that the support of $B$ is exactly $\bigcup_{i\in \NN} S_i\times T_i$. Simple examples of blocky matrices are zero matrices, all-$1$s matrices, and identity matrices. 

\begin{proposition}[{\rm\cite{liv95}}] \label{propcontractive}
If a boolean matrix $A$ satisfies $\normmultiplier{A}\le 1$, then $A$ is a
blocky matrix. In particular, either $A = 0$ and $\normmultiplier{A} = 0$,
or $A$ is a nonzero blocky matrix and $\normmultiplier{A} = 1$.
\end{proposition}
\begin{proof}
It is straightforward to verify that the Schur multiplier norm of every non-zero blocky matrix is $1$. To prove the converse, we must show that a boolean matrix with $\normmultiplier A \le 1$ does not contain a $2\times 2$ submatrix with exactly three $1$-entries. A direct calculation, due to L.~Livshits~\cite{liv95},
shows that $\normmultiplier{\matr1011} = 2/\sqrt 3 > 1$,
and the Schur multiplier norm does not increase upon restriction to a submatrix.
\end{proof}

Consider the algebra of functions $f: G\to \CC$ under addition and pointwise multiplication. This is a Banach algebra under the norm $\normalg{\cdot}$, since $\normalg{fg} \le \normalg{f} \cdot \normalg{g}$. The idempotent elements of this algebra are precisely the boolean functions $f: G\to \{0,1\}$, and
\Cref{propcontractive} is the matrix analogue of Kawada and It\^o's theorem that a function is idempotent and contractive if and only if it is the indicator of a coset.

In the same way that Cohen's idempotent theorem shows that any idempotent element of the Fourier algebra can be written as a finite sum of contractive idempotents, it is an open problem, first stated (to our knowledge) in~\cite[Section 2]{kp05}, whether any idempotent Schur multiplier can be written as a finite sum of contractive idempotents in the algebra of Schur multipliers. This question is considered one of the challenging open problems in the area (see~\cite{MR2777487} and~\cite[Question 3.13]{elt16}). A straightforward compactness argument, presented in \cite[Theorem~3.10]{hhh2023}, shows that a positive resolution to this problem, while only being meaningful for infinite matrices, is equivalent to the following conjecture for finite matrices.

\begin{conjecture}[{\rm\cite{kp05};} {\rm\cite{hhh2023},} Conjecture {\small III}\/]\label{conjblocky}
Suppose that $A$ is a finite boolean matrix
with $\normmultiplier{A} \le \gamma$. Then we may write
\begin{equation}A = \sum_{i=1}^L \sigma_i B_i,\end{equation}
where the $\sigma_i$ are signs, the $B_i$ are blocky matrices, and $L$ depends only on $\gamma$.
\end{conjecture}

\Cref{conjblocky} is closely related to the quantitative version of Cohen's idempotent theorem. Given a finite group $G$ and a function $f:G \to \CC$, define the matrix $M_f :  G\times G \to \CC$ by setting $M_f(x,y) = f(y^{-1}x)$. We have (see, e.g., \cite[Corollary~3.13]{hhh2023}) 
\begin{equation}\normmultiplier{M_f} = \normalg{f}.\label{eqgammaspectral}
\end{equation}
Therefore, it easily follows from the Green--Sanders theorem (and its extension to nonabelian groups by Sanders) that \Cref{conjblocky} is true for convolution matrices, i.e., those of the form $M_f(x,y) = f(y^{-1}x)$ for some boolean function $f:G \to \{0,1\}$.

We define the \emph{block complexity} $\block(A)$ of an $m\times n$ (integer) matrix $A$ to be the smallest integer $L$ such that there exist blocky matrices $B_1, \ldots, B_L$ and signs $\sigma_1, \ldots, \sigma_L$ with $A = \sum_{i=1}^L \sigma_i B_i$. 
D.~Avraham and A.~Yehudayoff~\cite{avrahamyehudayoff} proved that the block complexity of a random $n \times n$ boolean matrix is at least $n/\bigl(4 \log_2(2n)\bigr)$ with high probability. 

It is immediate from the definition of the block complexity that 
\[\normmultiplier{A} \le \block(A).\] 
\Cref{conjblocky} claims that, conversely, the block complexity can be bounded by a function of the Schur multiplier norm. In this paper, we prove that any boolean matrix with bounded Schur multiplier norm has block complexity at most polylogarithmic in its dimension.
All logarithms in this paper are understood to be natural, unless otherwise noted.

\begin{theorem}[Main theorem]\label{thmmain}
Let $A$ be an $m\times n$ integer matrix with $\normmultiplier{A} \le \gamma$. Denoting $k=\min\{m,n\}$, the block complexity of $A$ satisfies
\begin{equation}\block(A) \le 2^{O(\gamma^7)} \log\bigl(k\bigr)^2.\end{equation}
\end{theorem}

Note that, similar to Cohen's idempotent theorem, \cref{thmmain} holds more generally for integer matrices. 

\paragraph{The factorization norm.}\hskip-0.3em
In the proof of \Cref{thmmain}, we will mainly work with Grothen\-dieck's characterization of Schur multipliers through the $\gamma_2$ factorization norm.
The \emph{$\gamma_2$ factorization norm} (or the \emph{$\gamma_2$ norm} for short) of a complex matrix $A$ is defined as
\begin{equation}\normgamma A = \min_{UV = A} \normrow U \normcol V,\end{equation}
where the minimum is taken over all factorizations $UV$ of $A$,
$\normrow U$ is the largest $\ell^2$-norm of a row in $U$ and
$\normcol V$ is the largest $\ell^2$-norm of a column in $V$.
Note that the $\gamma_2$ norm cannot increase by restricting to a submatrix, since removing rows of $U$ or columns of $V$ cannot increase $\normrow U$ or $\normcol V$. In particular, $\normgamma A$ is at least $\normmax A = \max_{(x,y)\in X\times Y} \bigl|A(x,y)\bigr|$.

The factorization norm is germane to the study of the Schur multipliers, because it turns out that $\normmultiplier{A} = \normgamma A$ for every matrix $A$; this result is due to A.~Grothen\-dieck~\cite{grothendieck} (see also~\cite{pisier2001} for a proof in a more modern context). 



\paragraph{Definitions and notation.}\hskip-0.3em%
We shall often find it convenient to regard a matrix as a function with a set $X\times Y$ as its domain. In the rest of the paper, $X$ and $Y$ will be assumed finite unless otherwise stated, and we reserve in advance the variables $m = |X|$ and $n=|Y|$. Then, when we write $A\in \ZZ^{X\times Y}$ (for instance), we mean that $A$ is an $m\times n$ integer matrix with rows indexed by $X$ and columns indexed by $Y$.
If $A$ is a matrix with domain $X\times Y$ and let $X'\subseteq X$ and $Y'\subseteq Y$, we may write $A_{X'\times Y'}$ for the matrix on the domain $X'\times Y'$ with $A_{X'\times Y'}(x,y) = A(x,y)$ for all $x\in X'$ and $y\in Y'$.

Now suppose that $A$ is a complex matrix with $\normgamma A\le \gamma$.
By rescaling the vectors given by the definition of factorization norm, we may express $A$ as a product $A = UV$ for some matrices $U$ and $V$ such that
\begin{equation*}\max\bigl( \norm{u_1}_2,\ldots,\norm{u_m}_2\bigr) \le 1\end{equation*}
and
\begin{equation*}\max\bigl(\norm{v_1}_2,\ldots,\norm{v_n}_2\bigr) \le \gamma,\end{equation*}
where $u_1,\ldots, u_m$ are the rows of $U$ and $v_1,\ldots, v_n$ are the columns of $V$. We shall call such a factorization of $A$ a \emph{$\gamma$-factorization}. If $U$ is $m\times t$ and $V$ is $t\times n$ for some integer $t$, note that the set $\{u_1, \ldots, u_m, v_1, \ldots, v_n\}$ is contained in a $t'$-dimensional subspace of $\CC^t$ for some $t'\le m+n$. So, without loss of generality, we always assume that all the vectors in a $\gamma$-factorization of an $m\times n$ matrix belong to $\CC^{m+n}$.

\section{The Littlestone dimension}

The Littlestone dimension is a combinatorial parameter of sign matrices that characterizes regret bounds in online learning~\cite{littlestone1988}. Towards our development of a generalization of the Littlestone dimension for real matrices, we devote this section to recalling some properties of the ordinary Littlestone dimension.

We begin with the definition of the \emph{Vapnik--Chervonenkis (VC) dimension}. The VC dimension of a sign matrix $A \in \{-1,1\}^{X\times Y}$ is the size of the largest subset $X'$ of $X$ with the following property: for every $b : X'\to \{-1,1\}$, there exists a column $y$ such that $A(x,y) = b(x)$ for all $x\in X'$. When this condition holds, we say that $A$ shatters the set $X'$.

The Littlestone dimension relaxes this definition by shattering decision trees instead of sets. A \emph{mistake tree} of depth $d$ over a domain $X$ is a complete binary tree of depth $d$ in which
\begin{itemize}\setlength\itemsep{-2pt}
\item[i)] each internal node $\nu$ is labelled with an element $x(\nu)\in X$; and
\item[ii)] each edge $e$ is labelled with a sign $\sigma(e)\in \{-1,1\}$, where $\sigma(e) = -1$ indicates a left child and $\sigma(e) = 1$ indicates a right child.
\end{itemize}

For the purposes of the Littlestone dimension, a matrix $A \in \{-1,1\}^{X\times Y}$ is said to \emph{shatter} a mistake tree over $X$ if, for every root-to-leaf path $(\nu_1,\ldots, \nu_{d+1})$, there exists a column $y\in Y$ such that $A\bigl( x(\nu_i), y\bigr) = \sigma(\nu_i \nu_{i+1})$ for all $i\in [d]$. Then the \emph{Littlestone dimension} of $A$, denoted by $\Ldim(A)$, is the largest integer $d$ for which there exists a mistake tree of depth $d$ that is shattered by $A$. The Littlestone dimension is at least the VC dimension, since every shattered set $X' = \{x_1,\ldots,x_d\}$ gives rise to a shattered mistake tree of depth $d$ where all nodes at level $i$ are labelled with $x_i$.

For any sign matrix $A$ with $\Ldim(A) = d$ and any $0<\eps<1/2$, one can find a fixed sign vector that matches an $\eps^d$ proportion of the columns of $A$, with error at most $\eps$.

\begin{proposition}\label{propsubset}
Let $A\in \{-1,1\}^{X\times Y}$ be a matrix with $\Ldim(A) = d$ and let $\eps\in (0,1/2)$. There exists a function $\sigma : X\to \{-1,1\}$ and some subset $S\subseteq Y$ with $|S|\ge \eps^d|Y|$ such that
\begin{equation}\Pr_{y\in S} \bigl[ A(x,y) \ne \sigma(x)\bigr] \le \eps\end{equation}
for all $x\in X$.
\end{proposition}

\begin{proof}
Let $A\in \{-1,1\}^{X\times Y}$ be a matrix, and let $x\in X$ be an arbitrary row of $A$. Let
\begin{equation*}Y_x^- = \{y\in Y : A(x,y) = -1\}\qquad\hbox{and}\qquad Y_x^+ = \{y\in Y : A(x,y) = 1\}.\end{equation*}
By the definition of the Littlestone dimension, we have
\begin{equation}\Ldim(A) = 1+ \min\bigl\{ \Ldim(A_{X\times Y_x^-}), \Ldim(A_{X\times Y_x^+})\bigr\},\end{equation}
which implies that the Littlestone dimension must drop on at least one of $A_{X\times Y_x^-}$ or $A_{X\times Y_x^+}$.

If $S=Y$ and $\sigma(x) = \sgn\bigl( \sum_{y\in S} A(x,y)\bigr)$ does not satisfy the condition, then there is a row $x\in X$ with
\begin{equation}\Pr_{y\in S} \bigl[ A(x,y) \ne \sigma(x)\bigr] > \eps,\end{equation}
which implies that each of $Y_x^-$ and $Y_x^+$ comprises at most a $(1-\eps)$ proportion of $Y$. By the claim of the previous paragraph, removing all the columns in one of these two sets decreases the Littlestone dimension by $1$.

The proposition now follows by induction on $d$, since when $d=0$, each row must either contain only $-1$s or only $1$s, which forces all columns to be identical.
\end{proof}

Next, we show that for sign matrices, the Littlestone dimension provides a lower bound for (the square of) the $\gamma_2$ norm. The following proposition essentially follows from the classic mistake-bound analysis~\cite{rosenblatt1958} of the perceptron algorithm~\cite{mp43}.

\begin{proposition}\label{propldimbound}
Every $A\in \{-1,1\}^{X\times Y}$ satisfies
\begin{equation}\normgamma A \ge \sqrt{\Ldim(A)}.\end{equation}
\end{proposition}

\begin{proof}
Let $\gamma = \normgamma A$ and $d = \Ldim(A)$.
Consider a $\gamma$-factorization $A = UV$ of $A$, so that $\norm{u_x} \le 1$ for all $x\in X$ and $\norm{v_y} \le \gamma$ for all $y\in Y$.  By the definition of the Littlestone dimension, $A$ shatters a mistake tree of depth $d$ over $X$.
We shall recursively construct a root-to-leaf path $(\nu_1, \ldots, \nu_{d+1})$ that furnishes our desired bound. Let $\nu_1$ be the root of the tree and set $\sigma_1 = 1$. Now for all $1\le k\le d$, let $\nu_{k+1}$ be the child of $\nu_k$ with $\sigma(\nu_k\nu_{k+1}) = \sigma_k$
and set
\begin{equation}\sigma_{k+1} = -\sgn \biggl\langle \sum_{i=1}^k \sigma_i u_{x(\nu_i)}, u_{x(\nu_{k+1})}\biggr\rangle.\end{equation}
For each $1\le k\le d-1$, we have
\begin{align}
\normm{\sum_{i=1}^{k+1} \sigma_i u_{x(\nu_i)}}^2 &=
\normm{\sum_{i=1}^k \sigma_i u_{x(\nu_i)}}^2 + \bignorm{u_{x(\nu_{k+1})}}^2 + 2\sigma_{k+1} \biggl\langle \sum_{i=1}^k \sigma_i u_{x(\nu_i)}, u_{x(v_{k+1})}\biggr\rangle \cr
&\le \normm{\sum_{i=1}^k \sigma_i u_{x(\nu_i)}}^2 + 1,
\end{align}
by our choice of $\sigma_{k+1}$, and consequently
\begin{equation}\normm{\sum_{i=1}^d \sigma_i u_{x(\nu_i)}}^2 \le d.\end{equation}
By assumption, $A$ shatters the mistake tree, so there is a column $y\in Y$ with $A\bigl(x(\nu_i),y\bigr) = \sigma_i$ for all $1\le i\le d$. Hence,
\begin{equation} d = \sum_{i=1}^d A\bigl( x(\nu_i), y\bigr) \sigma_i
= \biggl\langle \sum_{i=1}^d \sigma_i u_{x(\nu_i)}, v_y\biggr\rangle
\le \sqrt d \norm{v_y} \le \gamma \sqrt d\end{equation}
by the Cauchy--Schwarz inequality.
\end{proof}

\section{A weighted Littlestone dimension for real matrices}

Because the inductive step in our eventual proof of \Cref{thmmain} does not preserve the boolean entries of the starting matrix, we must develop a generalization of the Littlestone dimension for real matrices. This we shall do in the present section, proving analogues of Propositions~\ref{propsubset} and~\ref{propldimbound}.

Let $\alpha > 0$ be a parameter. A \emph{weighted mistake tree} of depth $d$ over a domain $X$ is a complete binary tree of depth $d$, in which each internal node $\nu$ is labelled by an element $x(\nu)\in X$, as well as
a real parameter $w(\nu)\in \RR$. We also fix a choice of left child and right child for each of these internal nodes.
We say that $A\in \RR^{X\times Y}$ \emph{$\alpha$-shatters} a weighted mistake tree if for every root-to-leaf path $(\nu_1,\ldots, \nu_{d+1})$, there exists a column $y\in Y$ such that for all $1\le i\le d$,
\begin{itemize}\setlength\itemsep{-2pt}
\item[i)] whenever $\nu_{i+1}$ is the left child of $v_i$,
we have $A\bigl(x(\nu_i), y\bigr) \ge w(\nu_i) + \alpha/2$; and
\item[ii)] whenever $\nu_{i+1}$ is the right child of $v_i$, we have $A\bigl(x(\nu_i), y\bigr) \le w(\nu_i) - \alpha/2$.
\end{itemize}
Note that if $A$ $\alpha$-shatters a weighted mistake tree, then for every node $\nu$ in the tree, we must have $\bigl|w(\nu)\bigr| \le \normmax A$.
We define the \emph{$\alpha$-weighted Littlestone dimension}, denoted by $\Ldim_\alpha(A)$, to be the largest $d$ such that there exists a weighted mistake tree of depth $d$ that is $\alpha$-shattered by $A$.

The following proposition shows that for all real matrices $A$ with small $\Ldim_\alpha(A)$, there is a function $g$ such that the entries in each row $x$ are mostly within $2\alpha$ of $g(x)$, after restricting to a large subset of columns.

\begin{proposition}\label{propalphasubset}
Let $\alpha$ and $\eps$ be positive parameters. Suppose that $A\in \RR^{X\times Y}$ is a matrix with $\Ldim_\alpha(A) = d$
and $\normmax A = M$. There exists a subset $S\subseteq Y$ with 
\begin{equation}|S|\ge |Y| \biggl(\frac{\eps}{\lceil 2M/\alpha\rceil}\biggr)^d\end{equation}
and a function $g : X\to [-M,M]$ such that for every $x\in X$,
\begin{equation} \Pr_{y \in S}\Bigl[\bigl| A(x,y) - g(x)\bigr|\ge 2\alpha \Bigr]\le\eps.\end{equation}
\end{proposition}

\begin{proof}
If for every $x\in X$ there is some $k\in \NN$ with
\begin{equation}\biggl|\Bigl\{ y\in Y : A(x,y) \in \bigl[-M + (k-2)\alpha, -M+ (k+2)\alpha\bigr] \Bigr\}\biggr| \ge (1-\eps)|Y|,\end{equation}
then we can set $g(x) = -M+k\alpha$ and $S=Y$ to prove the assertion. So, assume that there is some $x\in X$ for which such a $k$ does not exist. Letting
\begin{equation} S_i = \Bigl\{ y\in Y : A(x,y) \in \bigl[-M + (i-1)\alpha, -M+ i\alpha\bigr] \Bigr\},\end{equation}
this assumption can be re-expressed as
\begin{equation}|S_{k-1} \cup S_k \cup S_{k+1} \cup S_{k+2}| < (1-\eps) |Y|\end{equation}
for all $k\in \NN$. 
Since $A(x,y)\in [-M,M]$ for all $y\in Y$, there must exist some $i\in \NN$ with $|S_i| \ge |Y| / \lceil 2M/\alpha \rceil$.
On the other hand, we know that
\begin{equation}|S_{i-1} \cup S_i \cup S_{i+1} \cup S_{i+2}| \le (1-\eps) |Y|,\end{equation}
so there must be some $j\in \NN$ with $|j-i|\ge 2$ such that
$|S_j| \ge \eps|Y| / \lceil 2M/\alpha \rceil$. Note that every $y\in S_i$ and $y'\in S_j$ satisfy $\bigl| A(x,y) - A(x,y')\bigr| \ge \alpha$, so we see that
\begin{equation}\Ldim_\alpha(A) \ge 1 + \min \bigl\{ \Ldim_\alpha(A_{X\times S_i}), \Ldim_\alpha(A_{X\times S_j})\bigr\}.\end{equation}
Hence we have $\Ldim_\alpha(A_{X\times S_i})\le d-1$
or $\Ldim_\alpha(A_{X\times S_j})\le d-1$, and may now perform induction on $d$.

When $d=0$, pick any $y_0\in Y$ and let $g(x) = A(x,y_0)$ for all $x\in X$. The fact that $A$ does not shatter any $\alpha$-weighted mistake tree of positive depth tells us that $A(x,y) \in \bigl( g(x) - \alpha, g(x) + \alpha)$ for all $(x,y)\in X\times Y$, and hence we may take $S=Y$. 
\end{proof}

Like the ordinary Littlestone dimension, our $\alpha$-weighted Littlestone dimension is boun\-ded polynomially in terms of the $\gamma_2$-norm (up to a constant factor that depends on $\alpha$).

\begin{proposition}\label{propldimalphabound}
Let $A\in \RR^{X\times Y}$ be a matrix with $\Ldim_\alpha(A) = d$ (for some $\alpha > 0$), $\normgamma A = \gamma$, and $\normmax A = M$. Then
\begin{equation} \gamma \ge \frac{\alpha\sqrt d}{2(M+1)} - 1,\end{equation}
and consequently
\begin{equation} d = O_\alpha(\gamma^4).\end{equation}
\end{proposition}

\begin{proof}
Consider a $\gamma$-factorization $A = UV$, where $U$ has rows $u_x$ indexed by $x\in X$ and $V$ has columns $v_y$ indexed by $y\in Y$. Fix a weighted mistake tree of depth $d$ that is $\alpha$-shattered by $A$.
For each internal node $\nu$ in the tree, we define
\begin{equation}\tilde u(\nu) = u_{x(\nu)} \oplus w(\nu)\end{equation}
(where, for a vector $u\in \RR^t$ and a real number $\beta$,
$u\oplus \beta$ denotes the vector in $\RR^{t+1}$ obtained by appending $\beta$ to $u$). Note that
\[\norm{\tilde u(\nu)} \le \norm{u_{x(\nu)}} + \normmax{A} \le M+1.\]

As in the proof of \Cref{propldimbound}, we recursively define a root-to-leaf path that supplies our final bound.
For the base case, let $\nu_1$ denote the root of the tree and set $\sigma_1 = 1$. Now suppose that $\nu_1,\ldots,\nu_k$ and $\sigma_1,\ldots,\sigma_k$ are defined, for some $1\le k\le d$.
We set
\begin{equation}
\label{eq:sigma_def}
\sigma_{k+1} = -\sgn \biggl\langle \sum_{i=1}^k \sigma_i \tilde u(\nu_i), \tilde u(\nu_{k+1})\biggr\rangle;\end{equation}
if $\sigma_{k+1} = 1$, let $\nu_{i+1}$ be the left child of
$\nu_i$, and if $\sigma_{k+1} = -1$, let $\nu_{i+1}$ be the right child of $\nu_i$.
By the definition of $\sigma_{k+1}$, we have
\begin{align*}
\normm{\sum_{i=1}^{k+1} \sigma_i \tilde u(\nu_{k+1})}^2 &=
\normm{\sum_{i=1}^k \sigma_i \tilde u(\nu_i)}^2 + \bignorm{\tilde u(\nu_i)}^2 + 2\sigma_{k+1} \biggl\langle \sum_{i=1}^k \sigma_i \tilde u(\nu_i), \tilde u(\nu_{k+1})\biggr\rangle \cr
&\le \normm{\sum_{i=1}^k \sigma_i \tilde u(\nu_i)}^2 + \bignorm{\tilde u(\nu_{k+1})}^2
\end{align*}
for all $1\le k\le d-1$,
which shows that
\begin{equation}\normm{\sum_{i=1}^d \sigma_i \tilde u(\nu_i)}^2
\le \sum_{i=1}^d \bignorm{\tilde u(\nu_i)}^2 \le d(M+1)^2.\end{equation}
On the other hand, since $A$ $\alpha$-shatters the tree, there exists a column $y\in Y$ such that
\begin{equation}\sigma_i\bigl(\langle u_{x(\nu_i)}, v_y\rangle - w(v_i)\bigr)
=\sigma_i\bigl(A(x(\nu_i), y) - w(v_i)\bigr)
\ge \frac\alpha2.\end{equation}
for all $1\le i\le d$.
Letting $\tilde v = v_y \oplus -1$, we may now bound
\begin{equation}
\biggl\langle \sum_{i=1}^d \sigma_i \tilde u(\nu_i), \tilde v \biggr\rangle = \sum_{i=1}^d \sigma_i \Bigl( \bigl\langle u_{x(\nu_i)}, v_y\bigr\rangle - w(\nu_i)\Bigr) \ge
\frac{\alpha d}{2},
\end{equation}
and by the Cauchy--Schwarz inequality,
\begin{equation}\label{eq:csbound}\frac{\alpha d}{2} \le
\normm{\sum_{i=1}^d \sigma_i \tilde u(\nu_i)}\cdot
\bignorm{\tilde v} \le \sqrt d(M+1)(\gamma + 1).
\end{equation}
Hence
\begin{equation}\gamma \ge \frac{\alpha\sqrt d}{2(M+1)} - 1,\end{equation}
and we also have
$d = O(\gamma^4/\alpha^2)$, since $\normmax A \le \normgamma A$.
\end{proof}
\goodbreak

\section{Three lemmas}

In this section, we establish three miscellaneous lemmas. We begin with a simple upper bound on the block complexity of an integer matrix in terms of the $\ell_1$ norms of its rows.

\begin{lemma}\label{lemblockcomplexitybound}
If $A\in \ZZ^{X\times Y}$ contains only nonnegative entries, then
\begin{equation}\block(A) \le \max_{x\in X} \sum_{y\in Y} \bigl| A(x,y)\bigr|.\end{equation}
More generally, every $A\in \ZZ^{X\times Y}$ satisfies
\begin{equation}\block(A) \le 2 \max_{x\in X} \sum_{y\in Y} \bigl| A(x,y)\bigr|.\end{equation}
\end{lemma}

\begin{proof}
The second statement follows from the first by splitting a matrix $A$ into $A_+ - A_-$, where both $A_+$ and $A_-$ are nonnegative, and then observing that $\block(A) \le \block(A_+) + \block(A_-)$.

Hence, we may assume that all entries of $A$ are nonnegative.
Let $X'$ be the set of nonzero rows of $A$. For every nonzero row $x\in X'$, choose $y(x)\in Y$ such that $A\bigl(x,y(x)\bigr) \ne 0$. The set $\bigl\{ \bigl( x,y(x)\bigr) : x\in X'\bigr\}$ defines a blocky matrix $B\in \{0,1\}^{X\times Y}$. Furthermore, all the entries in $A' = A-B$ are nonnegative, and
\begin{equation}\max_{x\in X} \sum_{y\in Y} A'(x,y) = \max_{x\in X} \sum_{y\in Y} A(x,y) - 1.\end{equation}
This argument may be iterated until $A$ is the zero matrix.
\end{proof}

Next, we prove a simple technical lemma, showing that if some set of bounded vectors has a large average vector $\hat v$, then subtracting $\hat v$ from many of the individual vectors significantly reduces their norms.

\begin{lemma}\label{lemsubtractaverage}
Let $v_1, \ldots, v_r$ be vectors in a Hilbert space, and let $\hat v = \ex_{i\in [r]} v_i$ denote their average. If  $\norm{v_i} \le \gamma$ for all $1\le i\le r$, and $\norm{\hat v} = c$, then
\begin{equation}S = \bigl\{ i\in [r] : \norm{v_i-\hat v}^2 \le \norm{v_i}^2 - c^2/2\bigr\},\end{equation}
satisfies $|S| \ge c^2 r / (2\gamma^2)$, and for every $i\in S$,
\begin{equation}\norm{v_i - \hat v}^2 \le \norm{v_i}^2 - \frac{c^2}{2}.\end{equation}
\end{lemma}

\begin{proof}
Observe that
\begin{equation}\ex_{i\in [r]} \norm{v_i - \hat v}^2 = \bigl(\ex_{i\in [r]} \norm{v_i}^2\bigr) - \norm{\hat v}^2 = \bigl(\ex_{i\in [r]} \norm{v_i}^2\bigr) - c^2.\end{equation}
Therefore,
\begin{equation}c^2 = \ex_{i\in[r]}  \bigl(\norm{v_i}^2 - \norm{v_i - \hat  v}^2 \bigr) \le \frac{c^2}{2} + \frac{|S|}{r} \gamma^2.\end{equation}
This gives
\begin{equation} \frac{|S|}{r} \ge \frac{c^2}{2\gamma^2},\end{equation}
which is what we wanted.
\end{proof}

For a boolean matrix $A\in \{0,1\}^{X\times Y}$ with $|Y| =n$, the third and final lemma of this section supplies a partition $ Y = \bigcup_{i=1}^k S_i$ such that 
\begin{itemize}\setlength\itemsep{-2pt}
\item[i)] for every $S_i$ there is a row $x_i$ such that $A(x_i,y)=1$ for all $y \in S_i$; and
\item[ii)] for every row $x$ and every $\delta>0$, there are at most $O_\delta(\log n)$ sets $S_i$ with
\begin{equation}
\label{eq:delta_dense}
\Pr_{y \in S_i}\bigl[A(x,y)=1\bigr] \ge \delta.\end{equation}
\end{itemize}
A more general version of this statement for integer-valued matrices follows from this (by replacing each integer row with several boolean rows, each representing a different $b\in \NN\setminus\{0\}$) but we shall in fact prove the integer version directly.

\begin{lemma}\label{lemfirstpartition}
Let $A\in \ZZ^{X\times Y}$ be an integer matrix with no all-zero columns. There exists a partition $Y = \bigcup_{i=1}^k S_i$, along with rows $x_1,\ldots,x_k\in X$ and nonzero integers $b_1\ldots, b_k\in \ZZ\setminus\{0\}$ such that
\begin{itemize}\setlength\itemsep{-2pt}
\item[i)] for every $1\le i\le k$, we have $A(x_i,y) = b_i$ for all $y\in S_i$; and
\item[ii)] for all $\delta > 0$, every $x\in X$, and every $b\in \ZZ\setminus \{0\}$, the number of indices $1\le i\le k$ for which
\begin{equation}\Pr_{y \in S_i}\bigl[ A(x,y) = b\bigr] \ge \delta\end{equation}
is at most $(\log |Y| + 1)/\delta$.
\end{itemize}
\end{lemma}
\goodbreak

\begin{proof}

Consider the following greedy algorithm.

\alg{Algorithm I}{Construct sets $S_i$} 
Given an integer matrix $A\in \ZZ^{X\times Y}$, this algorithm outputs a partition $Y = \bigcup_{i=1}^k S_i$, as well as rows $x_1,\ldots, x_k$ and nonzero integers $b_1, \ldots, b_k$.
\begin{itemize}\setlength\itemsep{-2pt}
\item[\textbf{I1.}] [Initialize.] Set $R\gets Y$ and $i \gets 1$.
\item[\textbf{I2.}] [Loop.] While $R\ne \emptyset$, repeat steps I3 and I4.
\begin{itemize}
\item[\textbf{I3.}] [Greedy choice.] Choose $(x_i, b_i)\in X\times \bigl( \ZZ \setminus \{0\}\bigr)$ that maximizes the size of
\[S_i = \bigl\{y \in R : A(x_i, y) = b_i\bigr\}.\]
\item[\textbf{I4.}] [Update.] Set $R \gets R\setminus S_i$ and increment $i\gets i+1$.
\end{itemize}
\end{itemize}
\goodbreak

The claim is that the output of Algorithm~I satisfies our two desired properties. It is clear that the sets $S_i$ form a partition of $Y$, and by their definition in step I3, they satisfy property (i).

To prove the second assertion, we assign the cost $c(y) = 1/|S_i|$ to each $y\in S_i$. Let $x\in X$ and $b\in \ZZ\setminus\{0\}$ be arbitrary, and consider the set
\begin{equation*}\{y_1,\ldots, y_t\} = \bigl\{ y\in Y : A(x,y) = b\bigr\},\end{equation*}
which we have ordered such that $c(y_1) \le \cdots \le c(y_t)$.

First we show that for every $1\le i\le t$,
\begin{equation} c(y_i) \le \frac{1}{t-i+1}.\end{equation}
Indeed, suppose towards a contraction that $c(y_i) > 1/(t-i+1)$. Let
$S_r$ be the member of our partition with $y_i\in S_r$, and suppose that $|S_r| = k$, so that $1/k = c(y_i) > 1/(t-i+1)$. The set
$S'= \{y_j \in Y : j\ge i\}$
has size $t-i+1 > k$, and $A(x,y) = b$ for all $y\in S'$, hence by the greedy choice made in step I3, the partition element $S_r$ containing $y_i$ must be at least as large as $S'$, and contain strictly more than $k$ elements. This is a contradiction.

Hence we bound
\begin{equation}\sum_{i=1}^k \Pr_{y\in S_i} \bigl[ A(x,y) = b\bigr] 
= \sum_{i=1}^k \frac{\bigl\{y\in S_i : A(x,y) = b\bigr\}}{|S_i|}
= \sum_{i=1}^t c(y_i) \le \sum_{i=1}^t \frac{1}{i} \le \log|Y| + 1,\end{equation}
and the number of indices $i$ with $\Pr_{y\in S_i} \bigl[ A(x,y) = b\bigr] \ge \delta$ can be at most $(\log|Y| + 1)/\delta$.
\end{proof}

We conclude this section by remarking that for $\delta \in [0,1/2)$, one cannot improve upon the logarithmic bound that appears in \Cref{lemfirstpartition}, at least without imposing further restrictions on $A$. To see this, randomly generate a matrix $A\in \{0,1\}^{n \times n}$ by independently setting the entries to $0$ or $1$ with probability $1/2$ each. A straightforward first moment argument shows that, with probability $1-o(1)$, for any partition of the columns into sets $S_1,\ldots,S_k$, there exists some row with density at least $1/2$ in every $S_i$ with $i \le (1/10)\log n$. Similarly, it is not hard to show that with probability $1-o(1)$, any partition satisfying (i) must consist of $k \ge (1/10)\log n$ sets.  

\section{The main argument}

We are now ready to prove our main theorem. Given a matrix $A\in \ZZ^{X\times Y}$ with $\normgamma A \le \gamma$, we wish to express it as a sum of matrices with small block complexity, which in turn gives a bound on the block complexity of $A$. Our strategy shall be to partition the columns $Y$ into subsets $S_i$ using \Cref{lemfirstpartition} and then 
define a new matrix $A'$ by averaging the columns $y$ belonging to each $S_i$. The block complexity of $A'$ can be bounded by invoking \Cref{lemblockcomplexitybound}, and \Cref{lemsubtractaverage} will help us show that $\normgamma{A-A'}^2$ decreases by a constant, making the argument suitable for iteration.

The catch is that subtracting averages takes us outside the realm of integer-valued matrices, so we must consider a slightly wider class of matrices. This motivates the following definition, which is our analogue of the almost integer-valued functions used in~\cite{greensanders}.
Given a matrix $A\in \RR^{X\times Y}$, let $A_\ZZ\in \ZZ^{X\times Y}$ denote the entrywise rounding of $A$ to the nearest integer matrix, where $b+1/2$ is rounded down to $b$ for all $b\in \ZZ$. For any real $\eps > 0$, we say that a real-valued matrix $A\in \RR^{X\times Y}$ is \emph{$\eps$-almost integer-valued} if $\normmax{A-A_\ZZ}\le \eps$.

The following lemma is the heart of our induction, allowing us to decrease the factorization norm.

\begin{lemma}[Key lemma]\label{lemkey}
Let $A\in \RR^{X\times Y}$ be a real-valued matrix with $\normgamma A = \gamma$. Suppose further that $A$ is $\eps$-almost integer-valued for $\eps = 2^{-20\gamma^2}$. If $A_\ZZ$ is not an all-zero matrix, then there exists a $2\eps$-almost integer-valued matrix $A'\in \RR^{X\times Y}$ such that
\begin{equation}\normgamma{A-A'}^2 \le \gamma^2 - \frac{1}{8}\end{equation}
and
\begin{equation}\block(A'_\ZZ) \le 2^{O(\gamma^7)} \bigl(\log |Y|\bigr)^2.\end{equation}
\end{lemma}

\begin{proof}
Let $M = \normmax A$. 
If $A_\ZZ$ is not an all-zero matrix, then $M \ge 1/2$ and we have $\gamma \ge 1/2$ as well. This implies that $\eps \le 2^{-10\gamma} \le 1/32$. We fix a $\gamma$-factorization $A = UV$ of $A$, with $\norm{u_x} \le 1$ for all $x\in X$ and $\norm{v_y}\le \gamma$ for all $y\in Y$.

By applying \Cref{lemfirstpartition} to (the nonzero columns of) $A_\ZZ$, we obtain rows $x_1,\ldots,x_k\in X$, integers $b_1,\ldots,b_k\in \ZZ\setminus\{0\}$, and sets $S_1,\ldots,S_k\subseteq Y$
such that for every $1\le i\le k$ and every $y\in S_i$, we have $A_\ZZ(x_i,y) = b_i$. (The second statement of \Cref{lemfirstpartition} will not be used until later, and we defer our choice of $\delta$ until then.)

We now invoke the properties of the $\alpha$-weighted Littlestone dimension. By \Cref{propldimalphabound}, we have $\Ldim_{1/8}(A)\le O(\gamma^4)$, and as a result, we also have $\Ldim_{1/8}(A_{X\times S_i}) = O(\gamma^4)$ for all $1\le i\le k$. Set $\eps_1 = \eps/(10\gamma)$ and apply \Cref{propalphasubset} to each matrix $A_{X\times S_i}$ with this parameter to obtain subsets $S_i'\subseteq S_i$ with
\begin{equation} |S_i'| \ge \Bigl( \frac{\eps_1}{\gamma}\Bigr)^{\!O(\gamma^4)}|S_i|
\ge  {\eps_1}^{O(\gamma^5)} |S_i|\end{equation}
and functions $g'_i: X\to [-M,M]$.  Rounding each of these functions $g_i'$ to their nearest integer functions, we obtain functions $g_i : X\to [-M-1/2,M+1/2] \cap \ZZ$ such that
\begin{equation}\Pr_{y\in S_i'}\bigl[ A_\ZZ(x,y)\ne g_i(x)\bigr] \le
\Pr_{y\in S_i'} \Bigl[ \bigl| A(x,y) - g_i'(x) \bigr| \ge 1/4 \Bigr] \le \eps_1\end{equation} 
for all $1\le i\le k$, where we have applied the fact that $A$ is $\eps$-almost integer-valued for some $0<\eps\le 1/32$. 

\begin{figure}[t]
\includegraphics[width=390pt]{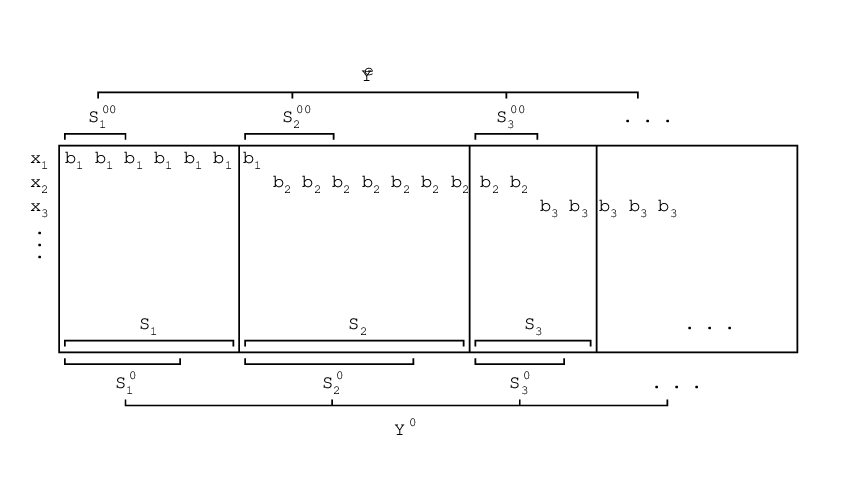}
\captionsetup{name=Fig.}
\vskip-20pt
\caption{A diagram illustrating various sets defined in the proof of \Cref{lemkey}.}
\end{figure}
For each $i$, let $\hat v_i$ denote the average $\ex_{y\in S_i'} v_y$,
and observe that since $A_\ZZ(x_i, y) = b_i \ne 0$,
\begin{equation}\bigl| \langle u_{x_i}, \hat v_i\rangle \bigr|
= \Bigl| \ex_{y\in S_i'} \langle u_{x_i}, v_y\rangle \Bigr|
= \Bigl| \ex_{y\in S_i'} A(x_i, y)\bigr|
\ge |b_i - \eps| \ge \frac{1}{2},\end{equation}
which in turn gives the bound
$\norm{\hat v_i} \ge 1/2$ by the Cauchy--Schwarz inequality.
We may then apply \Cref{lemsubtractaverage} to obtain subsets $S_i''\subseteq S_i'$ with
\begin{equation}|S_i''| \ge \frac{1}{8\gamma^2} |S_i'|\end{equation}
such that
\begin{equation}\norm{v_y - \hat v_i}^2 \le \norm{v_y}^2 - \frac{1}{8}\le \gamma^2 - \frac{1}{8}\end{equation}
for every $y\in S_i''$.
Let $Y' = \bigcup_{i=1}^k S_i'$ and $\tilde Y = \bigcup_{i=1}^k S_i''$;
we have
\begin{equation}\label{eqloselog}|\tilde Y| \ge \frac{1}{8\gamma^2} |Y'| \ge \frac{{\eps_1}^{O(\gamma^5)}}{8\gamma^2} |Y|
\ge 2^{-O(\gamma^7)}|Y|.\end{equation}
We now define a matrix $\hat A \in \RR^{X\times \tilde Y}$ by setting
\begin{equation} \hat A(x,y) = \ex_{y'\in S_i'} A(x,y') = \langle u_x, \hat v_i\rangle\end{equation}
for every $x\in X$ and $y\in S_i''$.
(Notice that the average is taken over $S_i'$ and not $S_i''$.)
Then, letting
$\tilde v_y = v_y - \hat v_i$ so that $\norm{\tilde v_y}^2 \le \gamma^2 - 1/8$, we have
\begin{equation}A(x,y) - \hat A(x,y) = \langle u_x, v_y - \hat v_i\rangle = \langle u_x, \tilde v_y\rangle\end{equation}
for every pair $(x,y)\in X\times S_i''$.

Next, we show that $\hat A$ is $2\eps$-almost integer-valued and bound the block complexity of $\hat A_\ZZ$. For every $x\in X$, the fact that $A$ is $\eps$-almost integer-valued, combined with our earlier bound $\Pr_{y\in S_i'}\bigl[ A_\ZZ(x,y)\ne g_i(x)\bigr] \le \eps_1$, shows that for all $(x,y)\in X\times S_i''$,
\begin{equation}\bigl| \hat A(x,y) - g_i(x)\bigr|
= \Bigl|\bigl(\ex_{y'\in S_i'} A(x,y') - g_i(x) \bigr)\Bigr|\le (1-\eps_1)\eps + \eps_1 (2M+1) \le \eps + (2\gamma+1)\eps_1\le  2\eps.\end{equation} 
That is, $\hat A$ is a $2\eps$-almost integer-valued matrix, and $\hat A_\ZZ(x,y) = g_i(x)$ for all $(x,y)\in X\times S_i''$. Define the matrix $G\in \ZZ^{X\times [k]}$ by $G(x,i) = g_i(x)$. Since $\hat A_\ZZ$ is obtained from $G$ by creating $|S_i''|$ copies of each column $i$, we have $\block(\hat A_\ZZ) = \block(G)$. Consider a fixed $x\in X$ and integer $b\ne 0$. If $g_i(x) = b$, then since $\Pr_{y\in S_i'}\bigl[ A_\ZZ(x,y) = b\bigr] \ge 1-\eps_1$, we have
\begin{equation}\Pr_{y\in S_i}\bigl[ A_\ZZ(x,y) = b\bigr] \ge   (1-\eps_1)\frac{|S_i'|}{|S_i|} \ge (1-\eps_1){\eps_1}^{O(\gamma^5)}\ge 2^{-O(\gamma^7)}.\end{equation}
It is now time to apply the second assertion of \Cref{lemfirstpartition}, with $\delta = 2^{-O(\gamma^7)}$. It tells us that there are at most $\bigl(\log|Y| + 1\bigr)/\delta = 2^{O(\gamma^7)} \log |Y|$ indices $i$ with $g_i(x) = b$. In other words, for any $x\in X$, each of the nonzero integers in the range $[-M,M]$ appears in the sum $\sum_{i=1}^k \bigl|g_i(x)\bigr|$ no more than $2^{O(\gamma^7)}\log|Y|$ times. There are no more than $2M$ nonzero integers in this range, and we have
\begin{equation}\sum_{i=1}^k \bigl|g_i(x)\bigr| \le 2M^2 \cdot 2^{O(\gamma^7)} \log|Y| = 2^{O(\gamma^7)}\log|Y|.\end{equation}
By \Cref{lemblockcomplexitybound}, we see that
\begin{equation}\label{eqtildeblockbound}\block(\hat A_\ZZ) = \block(G) \le 2\max_{x\in X} \sum_{i=1}^k \bigl|g_i(x)\bigr| \le 2^{O(\gamma^7)}\log|Y|.\end{equation}

The key observation at this juncture is that, moving from a $\gamma$-factorization $A = UV$ to a $\gamma$-factorization $\hat A = \hat U\hat V$, we may take $\hat U = U$; in other words, the row vectors $u_x$ remain unchanged. This allows us to repeatedly apply what we derived above to the remaining columns and concatenate all the matrices $\hat A$ into a single matrix $A'\in \RR^{X\times Y}$ with $\normgamma{A-A'}^2\le \gamma^2 - 1/8$.

In this spirit, starting with $Y_0 = Y$ and $\tilde Y_0 = \emptyset$, we construct a partition $Y = \bigcup_{i=1}^t \tilde Y_i$ and matrices $\hat A_i \in \RR^{X\times \tilde Y_i}$ by repeatedly applying the argument above to the decreasing sequence $Y_i = Y\setminus \bigcup_{j=0}^{i-1} \tilde Y_j$. Since $|Y_i| \le \bigl( 1-2^{-O(\gamma^7)}\bigr) |Y_{i-1}|$, this process stops after $t \le 2^{O(\gamma^7)} \log |Y|$ steps.

Let $A'\in \RR^{X\times Y}$ be the matrix with $A'_{X\times \tilde Y_i} = \hat A_i$ for all $1\le i\le t$. We have
\begin{equation}\block(A'_\ZZ) \le \sum_{i=1}^t \block\bigl( (\hat A_i)_\ZZ\bigr) \le t\cdot 2^{O(\gamma^7)} \log |Y| = 2^{O(\gamma^7)} \bigl(\log|Y|\bigr)^2.\end{equation}
By construction, the difference $A-A'$ may be factorized $A-A' = UV'$ where every column $v'_y$ in $V'$ satisfies $\norm{v'_y}^2\le \gamma^2-1/8$, so $\normgamma{A-A'}^2 \le \gamma^2 - 1/8$.
\end{proof}


Our main theorem now follows by iterating \Cref{lemkey}. In fact, we prove the following more general statement for almost integer-valued matrices.

\begin{theorem}[Generalized main theorem]\label{thmgeneralmain}
There is a constant $C$ such that the following holds. For every finite $\eps$-almost integer-valued matrix $A\in \RR^{X\times Y}$ with $\normgamma A\le \gamma$ and $\eps= 2^{-20\gamma^2}$, the block complexity of $A_\ZZ$ satisfies
\begin{equation}\block(A_\ZZ) \le 2^{C\gamma^7} \bigl(\log|Y|\bigr)^2.\end{equation}
\end{theorem}

\begin{proof}
If $\gamma < 1/2$, then as observed earlier, we have $\normmax A<1/2$ and $A_\ZZ$ is an all-zero matrix, which has block complexity zero. To prove the theorem, we bound the case $\normgamma{A}^2 \le \gamma^2$ in terms of the case $\normgamma{A}^2 \le \gamma^2 - 1/8$. Since decreasing $\gamma^2$ in this manner can only be done a finite number of times before it must satisfy $\gamma < 1/2$, the theorem follows by induction. 

If $\gamma \ge 1/2$, then by \Cref{lemkey}, there is a $2\eps$-almost integer-valued matrix $A'\in \RR^{X\times Y}$ with
\begin{equation}\normgamma{A-A'}^2 \le \gamma^2 - \frac{1}{8}\end{equation}
and
\begin{equation}\block(A'_\ZZ) \le 2^{C'\gamma^7} \bigl(\log |Y|\bigr)^2\end{equation}
for some absolute constant $C'$. Set $C= \max\{100C', 8^{7/2}\}$. 
The matrix $A-A'$ is $3\eps$-almost integer-valued and
\begin{equation}3\eps \le 3\cdot 2^{-20\gamma^2} \le 2^{-20(\gamma^2 - 1/8)},\end{equation}
so the induction hypothesis provides the bound
\begin{equation}\block\bigl( (A-A')_\ZZ\bigr) \le 2^{C(\gamma^2 - 1/8)^{7/2}}\bigl(\log|Y|\bigr)^2.\end{equation}
This yields
\begin{align}\block(A_\ZZ) &\le \block(A'_\ZZ) + \block\bigl( (A-A')_\ZZ\bigr) \cr
&\le \bigl(2^{C'\gamma^7} + 2^{C(\gamma^2 - 1/8)^{7/2}}\bigr) \bigl(\log |Y|\bigr)^2\cr
&\le  \bigl(2^{C(\gamma^2 - 1/8)^{7/2} + 1}\bigr)\bigl(\log |Y|\bigr)^2\cr
&\le2^{C\gamma^7} \bigl(\log|Y|\bigr)^2,\cr
\end{align}
where our assumption that $\gamma \ge 1/2$ implies
that $C \ge  C'\gamma^7/(\gamma^2-1/8)^{7/2}$.
\end{proof}

\paragraph{Concluding remarks.}\hskip-0.3em
There are two places in which we lose a logarithmic factor, resulting in the factor of $\bigl(\log|Y|\bigr)^2$ in our final theorem. The first is in \Cref{lemfirstpartition}, when we bound the number of indices by $\bigl(\log|Y|+1\bigr)/\delta$, and the second is in~\eqref{eqloselog}, where one discards a constant fraction of the columns. Tightening these bounds appears out of reach of our current methods; we suspect that novel ideas will be necessary to prove \Cref{conjblocky}. 

Recently, M.~G\"o\"os, N.~Harms, and A.~Riazanov proved~\cite{ghr2025} that for every $\eps > 0$, there exists an $n\times n$ boolean matrix $M$ and a real $\eps$-approximation $\tilde M$ of $M$ such that $\normgamma{\tilde M} \le O_\eps(1)$, while $\norm{M}_{\gamma_2} \ge 2^{\Omega_\eps\bigl(\sqrt{\log n}\bigr)}$ and hence $\block(M) \ge 2^{\Omega_\eps\bigl(\sqrt{\log n}\bigr)}$.  Their result implies that in \Cref{thmmain}, the hypothesis that $\eps$ is sufficiently small as a function of $\gamma$ is necessary, and that any proof of our theorem or of \Cref{conjblocky} must make essential use of the assumption that $A$ is integer-valued (or at least almost integer-valued).


\paragraph{Related work.}\hskip-0.3em
A recent paper of I.~Balla, L.~Hambardzumyan, and I.~Tomon approaches \Cref{conjblocky} from a different angle. To state their result, we introduce the following definition.
A \emph{monochromatic rectangle} in an integer matrix $A$ is the product of a subset $S$ of rows
and a subset $T$ of columns such that the entries $A(i,j)$ are all the same
for $(i,j)\in S\times T$.
We shall call a monochromatic rectangle all of whose entries are $b$ a \emph{$b$-rectangle}.

\begin{theorem}[{\rm\cite{bht2025}}, Theorem 1\/]\label{thmbht}
Suppose that $A$ is an $m\times n$ boolean matrix with
$\normgamma A\le \gamma$. There is a monochromatic rectangle $S\times T$ in $A$, where $S\subseteq [m]$ and $T\subseteq [n]$ satisfy
\begin{equation}\frac{|S|\cdot|T|}{mn} \ge 2^{-O(\gamma^3)}.\end{equation}
Specifically, if more than half of $A$'s entries are $1$, then $S\times T$ is a $1$-rectangle, and otherwise it is a $0$-rectangle.
\end{theorem}

A proof that \Cref{thmbht} would follow from \Cref{conjblocky} is presented as \cite[Lemma~3.5]{hhh2023}.
A subsequent paper of the authors uses \Cref{thmbht} to show that for any matrix $A$ with bounded $\gamma_2$ norm, one can find a large blocky matrix that contains a constant fraction of the $1$-entries in $A$.

\begin{theorem}[{\rm\cite{blockysubset2025}}, Theorem 1.4\/]
Let $A$ be an $m\times n$ boolean matrix $A$ with $\normgamma A\le \gamma$ in which the number of $1$-entries is $F$. There exists an $m\times n$ blocky matrix $B$ containing at least $F/2^{2^{O(\gamma)}}$ $1$-entries, such that for all $(i,j)\in [m]\times [n]$, we have $B(i,j) = 1$ only if $A(i,j) = 1$. 
\end{theorem}

\section*{Acknowledgements}

We would like to thank Ben Cheung, Nathan Harms, Rupert Levene, Vern Paulsen, and Ivan Todorov for helpful discussions. We are also indebted to the anonymous referees for valuable suggestions and for pointing out some minor errors that appeared in a preliminary version of the paper, including an issue with our earlier definition of $\alpha$-weighted Littlestone dimension.
Both authors are funded by the Natural Sciences and Engineering Research Council of Canada.
\bibliographystyle{alphacitation}
\xpatchcmd{\em}{\itshape}{\slshape}{}{}
\bibliography{citations}
\goodbreak
\bigskip\noindent
\textsc{Department of Mathematics and Statistics, McGill University, Montreal, Quebec, Canada}

\smallskip\noindent
\textsl{E-mail address}: \texttt{marcel.goh@mail.mcgill.ca}

\bigskip\noindent
\textsc{School of Computer Science, McGill University, Montreal, Quebec, Canada}

\smallskip\noindent
\textsl{E-mail address}: \texttt{hatami@cs.mcgill.ca}

\end{document}